\documentclass[a4paper, fleqn, abstract=true, headings=small, english, bibliography=totoc,12pt]{article}
\usepackage{setspace}
\usepackage{indentfirst} 

\usepackage[
  textwidth=16.9cm,  
  hmarginratio=1:1,   
  vmargin=2.3cm
]{geometry}

\usepackage[T1]{fontenc}
\usepackage[utf8]{inputenc}
\usepackage{csquotes}
\usepackage{bbm}

\usepackage{amsfonts,amsmath,amssymb,esint}
\usepackage{mathtools}
\usepackage{amsthm}
\usepackage{thmtools}

\usepackage{stmaryrd}

\usepackage[dvipsnames]{xcolor}

\usepackage{array}

\usepackage[alwaysadjust]{paralist}

\usepackage{babel}
\usepackage[theoremfont, osf]{newpxtext}
\usepackage[slantedGreek]{newpxmath}

\usepackage[shortcuts]{extdash}

\definecolor{purple}{cmyk}{0.75,0.90,0,0}
\definecolor{DB}{rgb}{0.07,0.0,0.5}
\definecolor{DG}{rgb}{0.0,0.37,0.07}%
\definecolor{DR}{rgb}{0.37,0,0.07}

\usepackage[
pdftitle={},%
pdfauthor={}, hyperindex=true, colorlinks=true, urlcolor={DR}, linkcolor={DR}, menucolor={DR}, citecolor={DR}, anchorcolor={DR},linktoc=all,pagebackref,final]{hyperref}

\theoremstyle{plain}
\newtheorem{proposition}{\protect\propositionname}[section]
\newtheorem{lemma}[proposition]{\protect\lemmaname}

\newtheorem{theorem}[proposition]{\protect\theoremname}

\usepackage{etoolbox}
\theoremstyle{definition}
\newtheorem{example}[proposition]{\protect\examplename}
\AtBeginEnvironment{example}{%
	\pushQED{\qed}%
}
\AtEndEnvironment{example}{\popQED\endexample}

\theoremstyle{remark}
\newtheorem{remark}[proposition]{\protect\remarkname}
\AtBeginEnvironment{remark}{%
	\pushQED{\qed}%
}
\AtEndEnvironment{remark}{\popQED\endremark}

\newtheorem{Step}{Step}
\newcounter{stepcount}
\numberwithin{Step}{stepcount}
\newcommand{\resetstep}{\stepcounter{stepcount}}

\declaretheoremstyle[%
spaceabove=.3\topsep,%
spacebelow=.3\topsep,%
headfont=\small\itshape,%
headpunct={. ---},%
notefont=\normalfont\itshape,%
bodyfont=\itshape,%
headindent=\parindent,%
numbered=yes,%
]{mystyle1}%

\newcounter{casecount}
\numberwithin{Case}{casecount}

\newenvironment{solution}{\noindent\emph{\protect\solutionname. }\pushQED{\qed}}{\popQED}

\newcommand*{\propositionname}{}
\newcommand*{\lemmaname}{}
\newcommand*{\sublemmaname}{}
\newcommand*{\theoremname}{}
\newcommand*{\corollaryname}{}
\newcommand*{\definitionname}{}
\newcommand*{\notationname}{}
\newcommand*{\claimname}{}
\newcommand*{\examplename}{}
\newcommand*{\remarkname}{}
\newcommand*{\exercisename}{}
\newcommand*{\openproblemname}{}
\newcommand*{\propertyname}{}

\newcommand*{\casename}{}
\newcommand*{\Claimname}{}
\newcommand*{\factname}{}

\newcommand*{\Casename}{}
\newcommand*{\solutionname}{}

\addto\captionsenglish{%
	\renewcommand{\propositionname}{Proposition}%
	\renewcommand{\lemmaname}{Lemma}%
	\renewcommand{\sublemmaname}{Sublemma}%
	\renewcommand{\theoremname}{Theorem}%
	\renewcommand{\corollaryname}{Corollary}%
	\renewcommand{\definitionname}{Definition}%
	\renewcommand{\notationname}{Notation}%
	\renewcommand{\claimname}{Claim}%
	\renewcommand{\examplename}{Example}%
	\renewcommand{\remarkname}{Remark}%
	\renewcommand{\exercisename}{Exercise}%
	\renewcommand{\openproblemname}{Question}%
	\renewcommand{\propertyname}{Property}%
	\renewcommand{\casename}{Case}%
	\renewcommand{\Claimname}{Claim}%
	\renewcommand{\factname}{Fact}%
	\renewcommand{\Casename}{Case}%
	\renewcommand{\solutionname}{Solution}%
}

\addto\captionsfrench{%
	\renewcommand{\propositionname}{Proposition}%
	\renewcommand{\lemmaname}{Lemme}%
	\renewcommand{\sublemmaname}{Sous-lemme}%
	\renewcommand{\theoremname}{Théorème}%
	\renewcommand{\corollaryname}{Corollaire}%
	\renewcommand{\definitionname}{Définition}%
	\renewcommand{\notationname}{Notation}%
	\renewcommand{\claimname}{Affirmation}%
	\renewcommand{\examplename}{Exemple}%
	\renewcommand{\remarkname}{Remarque}%
	\renewcommand{\exercisename}{Exercice}%
	\renewcommand{\openproblemname}{Question}%
	\renewcommand{\propertyname}{Propriété}%
	\renewcommand{\casename}{Cas}%
	\renewcommand{\Claimname}{Affirmation}%
	\renewcommand{\factname}{Fait}%
	\renewcommand{\Casename}{Cas}%
	\renewcommand{\solutionname}{Solution}%
}

\newcommand*{\intd}[1]{\mathop{}\!\diffd{#1}}
\newcommand*{\diffd}{{\operatorfont d}}

\newcommand*{\mE}{\mathscr{E}}
\newcommand*{\mF}{\mathscr{F}}

\newcommand*{\st}{\mathpunct{:}}

\newcommand*{\Lap}{\mathrm{\Delta}}

\newcommand{\ve}{\varepsilon}
\newcommand{\fo}{\forall\, }
\newcommand{\bp}{\begin{proof}}
	\newcommand\ep{\end{proof}}

\newcommand*{\cs}{_\textnormal{c}}
\newcommand*{\loc}{_\textnormal{loc}}

\def\R{{\mathbb R}}
\def\N{{\mathbb N}}

\newcommand{\verti}[1]{{\left\vert #1 \right\vert}}
\newcommand{\vertii}[1]{{\lVert #1 \rVert}} 
\newcommand{\vertiii}[1]
{{\left\vert\kern-0.25ex\left\vert\kern-0.25ex\left\vert #1 \right\vert\kern-0.25ex\right\vert\kern-0.25ex\right\vert}}
	\newcommand{\be}{\begin{equation}}
	\newcommand{\ee}{\end{equation}}
\newcommand{\bes}{\begin{equation*}}
	\newcommand{\ees}{\end{equation*}}
\newcommand{\ba}{\begin{aligned}{}}
	\newcommand{\ea}{\end{aligned}{}}

\DeclareMathOperator{\supp}{supp}

\DeclarePairedDelimiterX{\abs}[1]{\lvert}{\rvert}{\ifblank{#1}{\:\cdot\:}{#1}}

\DeclarePairedDelimiterX{\norm}[1]{\lVert}{\rVert}{\ifblank{#1}{\:\cdot\:}{#1}}

\DeclarePairedDelimiterX{\floor}[1]{\lfloor}{\rfloor}{\ifblank{#1}{\:\cdot\:}{#1}}

\DeclarePairedDelimiterX{\set}[2]{\lbrace}{\rbrace}{\ifblank{#2}{#1}{#1\st #2}}
\DeclarePairedDelimiterX{\ooInterval}[2]{\lparen}{\rparen}{#1,#2}
\DeclarePairedDelimiterX{\ccInterval}[2]{\lbrack}{\rbrack}{#1,#2}
\DeclarePairedDelimiterX{\coInterval}[2]{\lbrack}{\rparen}{#1,#2}
\DeclarePairedDelimiterX{\ocInterval}[2]{\lparen}{\rbrack}{#1,#2}

\DeclarePairedDelimiterX{\farg}[1]{\lparen}{\rparen}{\ifblank{#1}{\:\cdot\:}{#1}}
\DeclarePairedDelimiterX{\linfarg}[1]{\lbrack}{\rbrack}{\ifblank{#1}{\:\cdot\:}{#1}}
\DeclarePairedDelimiterX{\ScalarProd}[2]{\lparen}{\rparen}{\ifblank{#1}{\:\cdot\:}{#1} \mathbin\delimsize| \ifblank{#2}{\:\cdot\:}{{#2}}}
\DeclarePairedDelimiterX{\DualityProd}[2]{\langle}{\rangle}{\ifblank{#1}{\:\cdot\:}{#1},\ifblank{#2}{\:\cdot\:}{#2}}
\DeclarePairedDelimiterX{\FromTo}[2]{\lparen}{\rparen}{\ifblank{#2}{#1}{#1;#2}}

\numberwithin{equation}{section}

\newcommand\blfootnote[1]{%
  \begingroup
  \renewcommand\thefootnote{}\footnote{#1}%
  \addtocounter{footnote}{-1}%
  \endgroup
}

\begin{document}



\title{$L^2-$regularity of minimizers of anisotropic interaction functionals}
\author{Tristan Bullion-Gauthier \and Kai Xiao}
\date{\today}







\maketitle

\begin{abstract}
  We consider interaction functionals of the form 
  \begin{equation*}
  \nu\mapsto {\mathscr E}(\nu)=\int_{{\mathbb R}^N}\int_{{\mathbb R}^N} W(x-y)\intd\nu(y)\intd\nu(x)+\int_{{\mathbb R}^N} V(x)\intd\nu(x)\text{,}
  \end{equation*}
  involving an anisotropic  kernel $W$ and a general confinement potential $V$. Under standard assumptions on $W$ and its Fourier transform, we show that the  minimizer of ${\mathscr E}$ has $L^2$ density. Our argument relies on a careful analysis of nonlocal variational inequalities, which may be of independent interest.
\blfootnote{TB $\&$ KX. Universite Claude Bernard Lyon 1, CNRS,  Centrale  Lyon, INSA Lyon, Université Jean Monnet, ICJ UMR5208, 69622 Villeurbanne, France }
\blfootnote{TB \url{bullion@math.univ-lyon1.fr}}
\blfootnote{KX \url{xiao@math.univ-lyon1.fr}}
\blfootnote{Acknowledgments: We thank Petru Mironescu for introducing us to this research topic and for his interesting questions which are at the core of the present article. We are grateful to him for his careful reading and suggestions to improve the presentation. We also thank Sofiane Cherf for many insightful discussions. }
\blfootnote{Use of AI:  No generative AI was used in the preparation of this article.}
\end{abstract}



\section{Introduction} 
\noindent
\textit{Setting and main result.}
We study the regularity of minimizers of the functional given by 
\be \label{functional}
\mathscr{P}(\R^N) \ni \nu \mapsto  \mathscr{E}(\nu)\coloneq \int_{\R^N}\int_{\R^N}W(x-y) \intd\nu(y) \intd{\nu(x)} + \int_{\R^N}V(x) \intd\nu(x)\text{.}
\ee

Here, $N\ge 3$.
We assume that
\be\label{hv}\tag{HV}
 V \in (C^0\cap H^2_{\loc})(\R^N ; \R_+)\ \text{is a non-negative coercive function.} 
\ee
Concerning $W$, we assume that it is of the form
\be\label{H1} \tag{HW}
\ba
&W(x)=\frac{\Psi\left(x/\verti{x}\right)}{\verti{x}^{N-2}}\text{, where} \ \Psi \in C(\mathbb{S}^{N-1}) \ \text{is even and positive.}
\ea
\ee
We also assume that the Fourier transform of $W$ is of the form
\be\label{H2} \tag{HFW}
\ba
& \displaystyle \mathscr{F}(W)(\xi)= \frac{m(\xi/\verti{\xi})}{\verti{\xi}^2}\text{,} \ \ \text{where} \ m \in L^{\infty}(\mathbb{S}^{N-1}) \ \text{satisfies}
\\ & 0<m_0\le m(\omega)\le m_1<\infty\text{, } \fo\omega\in\mathbb{S}^{N-1}\text{.}
\ea
\ee 
 The fact that  $\mathscr{F}(W)$ is a $(-2)$-homogeneous tempered distribution is classical. When $N\geq 5$, $\mathscr{F}(W)$ belongs to $L^1_{\loc}$, and, in particular, $m$ in \eqref{H2} is an $L^1$ function. This is also the case when   $N=3, 4$ and $\Psi$ is sufficiently smooth (see, e.g.,  Lemoine~\cite[Theorem 3.2.4]{Lemoine1972Fourier}), but need not be  the case under the sole assumption \eqref{H1}.   The additional conditions that we impose on $m$ are essentially the same as in, e.g., Carillo \textit{et al.\ }\cite{carrillo2021equilibrium}, Mateu \textit{et al.\ }\cite{mateu2023explicit}, and Frank \textit{et al.\ }\cite{frank2026explicit}; see the discussion after the statement of Theorem \ref{main}. 

\smallskip
A typical example of a kernel $W$ satisfying \eqref{H1} and \eqref{H2} is
\bes
W_{\alpha}(x)\coloneq\frac{1}{\verti{x}^{N-2}}+ \alpha \frac{x_1^2}{\verti{x}^N}\text{, }\text{where }\alpha \in (-1,N-2)\text{,}
\ees
was considered in \cite{carrillo2021equilibrium}. It satisfies
\bes
\mathscr{F}(W_{\alpha})(\xi)=C_N\frac{(N-2-\alpha)\xi_1^2 + (N-2+\alpha)\sum_{k=2}^N \xi_k^2}{\verti{\xi}^4}>0\text{,}
\ees
where $C_N>0$ is a dimensional constant (see Stein~\cite [(32), p.73]{stein1970singular}). 

\smallskip
Other examples of kernels satisfying \eqref{H1} and \eqref{H2} 
are  \enquote {small} perturbations of the Coulomb kernel $1/|x|^{N-2}$, e.g.,
\bes
W(x)=\frac{1}{\verti{x}^{N-2}}+ \varepsilon \frac{\Omega(x/\verti{x})}{\verti{x}}, \text{where} \ \Omega \in C^{\infty}(\mathbb{S}^{N-1}) \ \text{and} \ \varepsilon>0 \ \text{is sufficiently small.}
\ees

\smallskip
Under the assumptions \eqref{hv}, \eqref{H1}, and \eqref{H2}, there exists a unique minimizer of the functional $\mathscr{E}$,  and this minimizer is compactly supported. In special cases, this was proved by  several authors, see, e.g., 
Serfaty \cite{serfaty2015coulomb}, Mora, Rondi, and Scardia~\cite{mora2019equilibrium}, Mateu \textit{et al.\ }\cite{mateu2023explicit}. 
For the convenience of the reader, we present a full proof of this fact in Appendix \ref{A}.

 \smallskip
Our main result is the following.
\begin{theorem}\label{main} Assume that \eqref{hv}, \eqref{H1}, and \eqref{H2} hold. Then, the unique minimizer $\mu$ of $\mE$  has $L^2$ density.
\end{theorem}

\noindent
\textit{Motivation and previous results.}
 The minimization of energies of the form \eqref{functional}, when $W$ is the Coulomb kernel ($-\log(\verti{x})$ in two dimensions), is a classical problem of electrostatics.
 But these energies also appear in  other contexts. For example, they arise in the study of the mean-field description of Coulomb gases, an important model in statistical mechanics (see \cite{serfaty2015coulomb}). They also play an crucial role in mathematical biology for their connection to aggregation equations (central in biological modelling, see Topaz, Bertozzi, and Lewis~\cite{topaz2006nonlocal}). This connection motivated  Carillo \textit{et\ al.\ }to study the existence of minimizers of \eqref{functional}, as well as their uniqueness and regularity properties in, e.g., Balagué \textit{et al.\ }\cite{balague2013dimensionality}, Carrillo, Delgadino, and Mellet~\cite{carrillo2016regularity}, Carrillo and Shu~\cite{carrillo2023radial}.  

\smallskip
The first results on the minimizers of \eqref{functional} were obtained in the isotropic case $W(x)=1/\verti{x}^{N-2}$, starting with the seminal works of Frostman \cite{frostman1935potential}.
 Recently, several explicit results were obtained in the case of anisotropic repulsive interactions $W$, following the important contributions of Mora, Rondi, and Scardia~\cite{mora2019equilibrium} and Carrillo \textit{et al.\ }\cite{carrillo2020ellipse}, where special perturbations of Coulomb interactions were considered in two dimensions. For example, in Mateu \textit{et al.\ }\cite{mateu2023explicit} and Carrillo and Shu~\cite{carrillo2024minimizers}, it is shown that, for $N=3$, $V(x)=|x|^2$, and under assumptions on $W$ in the spirit of \eqref{H1} and \eqref{H2}, the unique minimizer of $\mE$ is the normalized characteristic function of an ellipsoid.    Moreover, it is proved in \cite{mateu2023explicit} and \cite{carrillo2024minimizers} that, if the assumption \eqref{H2} is relaxed to $m\ge 0$,  then the minimizing  measure $\mu$ may be singular and thus have no density.

\smallskip
 These results were obtained \textit{via} the following strategy: (i) prove that minimality is equivalent to the Euler--Lagrange equation associated with $\mE$; (ii) find a solution of the Euler--Lagrange equation within the class of centered ellipsoids. While the first part of this program is robust and can be completed for large classes of potentials $V$ and $W$, the second one is specific to the case where $V$ is a quadratic form, typically $V(x)=|x|^2$, and its conclusion need not hold for a general $V$ (see, e.g., \cite[(2.18)]{mateu2023explicit}). 

\smallskip
In a different direction,  the case of a general $V$ when $W(x)=1/\verti{x}^{N-2}$ was studied in \cite{carrillo2016regularity}, where it is shown that minimizing $\mE$ amounts to solving an obstacle problem. This is exploited in order  to derive regularity results. The starting point in \cite{carrillo2016regularity} is the Euler--Lagrange equation associated with the minimization of $\mE$, which formally reads as
\be
\label{ELeq}
\begin{cases}
 W\ast \mu + 1/2 V  \geq \alpha \coloneq \mE(\mu) - 1/2 \int V \intd\mu\text{,}
\\
W\ast \mu + 1/2 V  = \alpha \ \text{on}  \ \supp \mu\text{,}
\end{cases}
\ee
if $\mu$ is the minimizer of $\mE$.

\smallskip
In the case where $W=1/\verti{x}^{N-2}$, using the above Euler--Lagrange equation, it is proved in \cite{carrillo2016regularity} that the potential $u=W\ast\mu$ satisfies the following obstacle problem 
\be\label{obs int}
\begin{cases}
& u \geq \alpha-1/2 V\text{,}
\\ &-\Lap u \geq 0 \ \text{in} \ \R^N\text{,}
\\ &-\Lap u=N(N-2)\abs{B_1} \mu= 0 \ \text{in} \ \{ u> \alpha -1/2V\} \subset \left(\supp \mu \right)^\mathrm{c}\text{.}
\end{cases}
\ee
This problem is associated with a \textit{variational inequality} 
\be \label{var int}
\ba
& w \in H\text{, } w \geq \alpha -1/2 V\text{,}  
\\& \int_{\Omega} \nabla w \cdot \nabla (v-w) \intd{x} \geq 0\text{,} \ \fo v \in H\text{,} \ v \geq \alpha -1/2V\text{,}
\ea
\ee
for an appropriate domain 
 $\Omega$ and $H^1$-type space $H$. 

\smallskip
While it is rather straightforward to obtain 
 \eqref{ELeq} and to derive from it \eqref{obs int}, getting \eqref{var int} form \eqref{obs int} requires more work. This program has been completed in \cite{carrillo2016regularity}  when $W$ is the Coulomb potential. One of the crucial steps in their approach consists of establishing the continuity of $W\ast\mu$; this relies on the fact that the Coulomb kernel is, up to a constant, a fundamental solution of the Laplacian.  For another approach in this case, see \cite[Proposition 2.22]{serfaty2015coulomb}, and also  Armstrong, Serfaty, and Zeitouni~\cite[Lemma 2.3]{armstrong2014remarks}. Once the validity of \eqref{var int} is established, one can appeal to 
 the regularity theory for obstacle problems and variational inequalities, which has been developed since the 1960's  starting with the works of Lewy, Stampacchia, Brezis, Kinderlehrer, and Frehse. As an example of use of this theory,  it is proved in  \cite[Theorem 3.4]{carrillo2016regularity} that, if $V$ is $W^{2,p}_{\loc}$ with $p>N$, then $\mu$ is $L^{\infty}$. 
 
\smallskip
\noindent
\textit{Comments on the proof of Theorem \ref{main}.}
We rely on the link between the minimization of $\mE$ and variational inequalities described above. In order to consider anisotropic kernels of the form \eqref{H1}, we are led to consider operators and linear or bilinear forms that are defined \textit{via} the Fourier transform.
Two main difficulties arise in this process. 

\smallskip
We face a first difficulty when trying to define a variational inequality involving the potential $u=W\ast\mu$. When $W=1/\verti{x}^{N-2}$, $u$ enjoys  properties that need not hold for a general $W$. More specifically, $u$ is superharmonic, and this allows the authors of \cite{carrillo2016regularity} to prove that $u$ is everywhere continuous. Superharmonicity is also used in the approximation procedure in  \cite[Proposition 2.22]{serfaty2015coulomb}.  This property is not available under the assumptions of Theorem \ref{main}. In order to bypass this difficulty, we rely on  regularity results established in Section \ref{lemmas}, results that may be of independent interest. 

\smallskip
The second difficulty arises from the fact that the variational inequality associated with the minimization of $\mE$ involves linear and bilinear forms defined in terms of the Fourier transform.  We cope with such forms using  a penalization approach strongly inspired by Kinderlehrer and Stampacchia~\cite[Chapter IV, Section 5]{kinderlehrer2000introduction}. This leads to the  study of nonlocal penalized problems different from the ones considered in \cite[Chapter IV, Section 5]{kinderlehrer2000introduction}. This requires a careful adaptation of the approach  in  \cite{kinderlehrer2000introduction}. An advantage of our approach when compared to the one in \cite{carrillo2016regularity} is that it does not rely on continuity and thus holds in any dimension $N\ge 3$. A drawback is that it only leads to  $L^2$ estimates, while one would expect $L^p$-regularity for some $p>2$ (depending on $N$). 

\smallskip
Our paper is organized as follows. In Section \ref{lemmas}, we gather a few auxiliary results and preliminary regularity results. In Section \ref{sec var}, we prove that  the potential $u=W\ast \mu$ satisfies a variational inequality. Finally, Section \ref{sec reg} is devoted to the proof of a regularity result on the solution of this variational inequality, and of Theorem \ref{main}.

\section{First regularity results}\label{lemmas}
In the following, we denote by $\mu$ the minimizer of $\mE$ and, with no loss of generality, assume that $\supp \mu \subset B_1$, where, for a non-negative measure $\nu$, 
\bes
\supp\nu \coloneq \{ x \in \R^N\text{;}\, \nu(B(x,r))>0\text{,} \ \fo r>0 \}\text{.}
\ees
 We will only consider spaces of real-valued functions and denote $L^2=L^2(\R^N; \, \R)$, and similarly for $L^p$, $\mathring{H}^1$, etc.
We will denote by $C$ any positive constant depending only on $N$ and $W$. We will sometimes use the notation \enquote{$A\lesssim B$} (respectively \enquote{$A\approx B$}) to indicate that $A\leq C B$ (respectively $\frac{1}{C} B\leq A \leq C B$) for such a constant $C$.

We start with a few standard results (Lemmas \ref{obs eq}--\ref{almost c0}). The first result is in the spirit of \cite[Theorem 3.1]{mora2019equilibrium}
\begin{lemma}
	\label{obs eq}
	$\mu$ satisfies
	\begin{align}
	  &\label{el1}W\ast \mu + 1/2 V  \geq \alpha \coloneq \mE(\mu) - 1/2 \int V \intd\mu\text{,} \ \text{Lebesgue a.e.,}
	\\&  \label{el2} W\ast \mu + 1/2 V  = \alpha\text{,} \ \mu-\text{a.e.,}
	\\ &\label{el3}W\ast \mu(x) + 1/2 V(x)  \leq  \alpha\text{,} \ \fo x \in \supp\mu\text{.}
	\end{align}
\end{lemma}
\begin{proof}
	Equations \eqref{el1} and \eqref{el2} follow from the proof of \cite[Theorem 3.1]{mora2019equilibrium}. Equation \eqref{el3} is a consequence of \eqref{el2}, thanks to the l.s.c.\ of $W\ast\mu +1/2V$, as shown in \cite[Theorem 4]{balague2013dimensionality}. 
\end{proof}
The next result is straightforward and we omit its proof.
\begin{lemma}\label{no dirac}
	Let $\nu$ be a non-negative measure such that
	\bes
	\int_{\R^N} W\ast \nu \intd \nu<\infty\text{.}
	\ees
	Then, $\nu (\{x\})=0$ for each $x\in \R^N$.
\end{lemma}
We next prove a first regularity result, which is a straightforward consequence of Landkof \cite[Lemma 1.7]{landkof1972foundations}.
We set  $\displaystyle E(x) \coloneq {1}/{\verti{x}^{N-2}}$.

\begin{lemma}\label{fun H1}
Let $\nu$ be a non-negative and finite measure such that 
\bes
\int_{\R^N} W\ast \nu \intd \nu<\infty\text{.}
\ees
Then, $W\ast \nu \in \mathring{H}^1\coloneq \{u \in L^{2N/(N-2)}\text{;} \, \nabla u \in L^2 \}$. 
\end{lemma}
\begin{proof}
 We have 
\be\label{land}
\int_{\R^N} \verti{\nabla(E\ast \nu)}^2 \intd{x} \leq C\int_{\R^N} E\ast \nu \intd{\nu}  \leq C \int_{\R^N}W\ast \nu \intd{\nu}<\infty\text{,}
\ee
where the first inequality follows from \cite[Lemma 1.7]{landkof1972foundations}, and the second inequality is obtained thanks to \eqref{H1}. 
The conclusion follows by approximation of $\nu$ with the convolutions $\nu_{\ve}\coloneq \nu \ast \rho_{\ve}$, where $\rho$ is a standard mollifier. Note that, since $W\in \mathscr{S}'$ and $\nu_{\ve} \in \mathscr{S}$, we have $\mathscr{F}(W\ast\nu_{\ve})=\mathscr{F}(W) \mathscr{F}(\nu_{\ve})$ in the sense of tempered distributions and thus $\mathscr{F}(W\ast\mu_{\ve})$ is a function. For each $\ve>0$, we have
\bes
\ba
 \frac{1}{(2\pi)^N}\int_{\R^N}\verti{\xi}^2\verti{\mathscr{F}(W\ast \nu_{\ve})(\xi)}^2 \intd\xi &= \int_{\R^N}  \verti{\xi}^2\verti{\mathscr{F}(W)(\xi)}^2 \verti{\mathscr{F}(\nu_{\ve})(\xi)}^2  \intd\xi \\ &\leq C \int_{\R^N}\verti{\xi}^2\verti{\mathscr{F}(E)(\xi)}^2 \verti{\mathscr{F}(\nu_{\ve})(\xi)}^2  \intd\xi 
\\ &= (2\pi)^N \int_{\R^N}\verti{\nabla(E\ast \nu_{\ve})}^2 \intd{x}
\\ & \leq (2\pi)^N \int_{\R^N}\verti{\nabla(E\ast \nu)}^2 \intd{x}\text{.}
\ea
\ees
Here, we rely on the fact that $m$ (defined in \eqref{H2}) is bounded for the first inequality. The last identity follows from Plancherel's theorem, while the last inequality is a standard convolution estimate. By \eqref{land}, we thus find that $\nabla(W\ast\nu_{\ve})$ is bounded in $L^2$ and this implies that $\nabla(W\ast \nu)$ is in $L^2$, since $W\ast \nu_{\ve} \to W\ast \nu$ in $\mathscr{S}'$. 

In order to show that $W\ast \nu$ is in $L^{2N/(N-2)}$, we note that $W= W\mathbbm{1}_{B_1}+W\mathbbm{1}_{(B_1)^\mathrm{c}} \in L^1+ L^{p}$ for any $p>N/(N-2)$. Since $\nu$ is a finite measure, we also have $W\ast \nu \in  L^1+ L^{p}$. Hence, $\{x \in \R^N; \, W\ast \nu(x) >t\}$ is of finite Lebesgue measure for each $t>0$ ($W\ast \mu$ is \enquote{vanishing at infinity}) and the Sobolev embedding (see, e.g., Leoni \cite[Theorem 12.4]{leoni2017first}) yields $W\ast \nu \in L^{2N/(N-2)}$. 
\end{proof}

	

\begin{lemma}(\cite[Theorem 1.4]{landkof1972foundations}) \label{Enu}
	Let $\nu$ be a non-negative measure. Then 
	\bes
	E\ast \nu(x) \geq \fint_{B(x,r)} E\ast \nu (z) \intd{z}\text{,} \ \fo x \in \R^N\text{,} \ r>0\text{.} 
	\ees
\end{lemma}
%
%
We now turn to an important auxiliary result in the vein of Evans' theorem (see, e.g., Caffarelli \cite[Theorem 1]{caffarelli1998obstacle}). Its conclusion, combined with Lemma \ref{Linfty}, leads to a variational inequality that will allow us to complete the proof of Theorem \ref{main}.
\begin{lemma}\label{almost c0} We have
	\begin{align}
		& \label{c1}W\ast\mu \ \text{is continuous in} \ (\supp \mu)^\mathrm{c}\text{.}
	\\ &\label{c2}W\ast \mu \ \text{is continuous at every} \ x \in \supp \mu \ \text{such that} \  W\ast\mu (x) +1/2 V(x)=\alpha\text{.}	
	\\
	&\label{c3h}W\ast \mu \ \text{is continuous} \ \mu-\text{a.e.\ in}\ \supp \mu\text{.}
	\end{align}
\end{lemma}

In the proof of Lemma \ref{almost c0} (ii), we will use the fact that $E\ast \mu$ is harmonic (in the classical sense) in $(\supp\mu)^\mathrm{c}$ (which is an open set), and therefore 
\be\label{mean val eq}
E\ast \mu(x) = \fint_{B(x,\text{dist}(x,\supp\mu))} E\ast\mu(z) \intd{z}
\ee 
for each $x \in (\supp\mu)^\mathrm{c}.$
\begin{proof}[Proof of Lemma \ref{almost c0}]
 \textit{Proof of	\eqref{c1}.}  For each $x \in (\supp\mu)^\mathrm{c}$ and each sequence $x_n \to x$, there exists $C_x<\infty$, such that, for $n$ sufficiently large,
	\bes
	0 \leq W(x_n-y) \leq C_x <\infty\text{,} \ \fo y \in \supp \ \mu\text{.} 
	\ees
	This inequality, combined with the fact that $W$ is continuous in $\R^N \setminus \{0\}$, yields the desired conclusion by dominated convergence.
	
\smallskip
\noindent
\textit{Proof of
	\eqref{c2}.} Let \(x\in\supp\ \mu\) be such that $W\ast\mu (x) +1/2 V(x)=\alpha.$ Let $x_n \to x$. We first observe that, by l.s.c.\ of $W\ast \mu$, we have 
	\be\label{f1}
	\lim_n W\ast \mu (y_n) = W\ast \mu(x)\text{,}
	\ee
	for each $(y_n) \subset \supp \mu$ converging to $x$. Indeed, we have
	\bes
	\ba
	\alpha-1/2V(x)=W\ast \mu (x) \leq \liminf_n W\ast \mu (y_n) &\leq \limsup_n W\ast \mu (y_n) \\ &\leq \lim_n \left(\alpha -1/2V(y_n)\right) = \alpha-1/2V(x)\text{,}
	\ea
	\ees
	using the l.s.c.\ of $W\ast \mu$ for the first inequality and \eqref{el3}.
	
	Thus, in order to prove \eqref{c2}, we only need to consider $x_n \to x$ with $(x_n) \subset (\supp \mu)^\mathrm{c}$. It suffices to check that 
	\be\label{bad part}
	\fo \ve >0\text{,} \ \exists \ \delta>0\text{, } \int_{B(x,\delta)} W(x_n- y) \intd\mu(y) <\ve\text{,} \ \fo n \ \text{sufficiently large}\text{.}
	\ee
	Indeed, granted \eqref{bad part}, we have that, for each $\ve>0$, there exists $\delta>0$ such that
	\bes
	\ba
	\verti{W\ast\mu(x_n) - W\ast\mu (x)}  &\leq \verti{\int_{B(x,\delta)^\mathrm{c}} W(x_n-y) \intd\mu(y) - \int_{B(x,\delta)^\mathrm{c}} W(x-y) \intd\mu(y) }
	\\ & \ \ \ \ + 2\ve\text{,}
	\ea
	\ees
	for $n$ sufficiently large. On the other hand,  for $n$ sufficiently large, we have $x_n \in B(x,\delta/2)$ and thus $\verti{x_n-y}\geq \delta/2$, for each $y \in B(x,\delta)^\mathrm{c}$. Using the dominated convergence theorem, we therefore have
	\be \label{to pass}
	\lim_n \verti{\int_{B(x,\delta)^\mathrm{c}} W(x_n-y) \intd\mu(y) - \int_{B(x,\delta)^\mathrm{c}} W(x-y) \intd\mu(y)} \to 0\text{.}
	\ee
	Hence, passing to the limits in \eqref{to pass}, we find 
	\bes
	\limsup_n \verti{W\ast\mu(x_n) - W\ast\mu (x)} \leq 2\ve\text{,} \ \fo \ve>0\text{,}
	\ees
	and this implies \eqref{c2}.
	
\smallskip
\noindent
\textit{Proof of	 \eqref{bad part}.} Let $\ve>0$. By dominated convergence, there exists $\delta>0$ sufficiently small such that
	\be\label{delta ch}
	\int_{B(x,\delta)} W(x-y) \intd\mu(y) < \ve\text{.}
	\ee
	Indeed, by \eqref{el3}, we have $W(x-\cdot) \in L^1(\mu)$. On the other hand, $\mathbbm{1}_{B(x,\delta)}(y) \to 0$ for $\mu$-a.e.\ $y \in \R^N$ since $\mu(\{x\})=0$ by Lemma \ref{no dirac}, which implies \eqref{delta ch}.
	
	For each $n \in \N$, set $r_n\coloneq \text{dist}(x_n,\supp \mu)$ and let $y_n\in \supp \mu$ be a point such that $\verti{y_n-x_n}=r_n$. We have
	\be\label{estimates}
	\ba
	\int_{B(x,\delta)} W(x_n-y)\intd\mu(y) &\overset{(a)}{\lesssim}  \int_{B(x,\delta)} E(x_n-y) \intd\mu(y) = E\ast\mu_{| B(x,\delta)}(x_n) 
	\\ & \overset{(b)}{=} \fint_{B(x_n,r_n)} (E\ast\mu_{|B(x,\delta)})(z) \intd{z}
	\\ &\leq 2^N \fint_{B(y_n,2r_n)} (E\ast\mu_{|B(x,\delta)})(z) \intd{z}
	\\ &\overset{(c)}{\lesssim}  (E\ast\mu_{|B(x,\delta)})(y_n)
	\\ & \overset{(d)}{\lesssim} \int_{B(x,\delta)} W(y_n-y) \intd\mu(y) 
	\\ & = W\ast\mu(y_n) - \int_{B(x,\delta)^\mathrm{c}} W(y_n - y) \intd\mu(y)\text{.}
	\ea
	\ee
	Here : (a) and (d) follow from \eqref{H1}; (b) follows from the fact that $E\ast \mu_{|B(x,\delta)}$ is harmonic and continuous outside of $\supp\mu$; (c) is a consequence of Lemma \ref{Enu}. 
	By \eqref{f1}, we have 
	\be\label{by f1} \lim_n W\ast\mu(y_n) = W\ast\mu(x)= \int_{B(x,\delta)} W(x-y) \intd\mu(y) + \int_{B(x,\delta)^\mathrm{c}} W(x-y) \intd\mu(y)\text{.}
	\ee
	On the other hand, we have 
	\be\label{by cvd}
	\verti{\int_{B(x,\delta)^\mathrm{c}} W(y_n-y) \intd\mu(y) - \int_{B(x,\delta)^\mathrm{c}} W(x-y) \intd{\mu(y)}} \to 0\text{,} \ \text{as} \ n \to \infty\text{.}
	\ee
	by dominated convergence, as in \eqref{to pass}. 
	
	Combining \eqref{estimates}, \eqref{by f1} and \eqref{by cvd}, we find
	\bes
	\ba
	\int_{B(x,\delta)} W(x_n-y) \intd\mu(y) &\lesssim W\ast\mu(x) - \int_{B(x,\delta)^\mathrm{c}} W(y_n - y) \intd\mu(y) +o_{n\to \infty}(1)
	\\ &=  \int_{B(x,\delta)} W(x-y) \intd\mu(y) + o_{n\to \infty}(1)\text{,}
	\ea
	\ees
	which yields \eqref{bad part}, by \eqref{delta ch}.
	\qedhere
	
\end{proof}
Similar arguments lead to the following result.
\begin{lemma}\label{Linfty}
	$W\ast\mu $ is bounded (everywhere).
\end{lemma}
\begin{proof}
	For each $x \in \supp \mu \subset B_1$, by \eqref{c3h}, we have
	\bes
	0 \leq (W\ast\mu)(x)\leq  \alpha-\frac{1}{2}V(x) \leq \alpha + \frac{1}{2} \vertii{V}_{L^{\infty}(B_1)}\text{.}
	\ees
	For $x \in (\supp \mu)^\mathrm{c}$, we denote $r\coloneq\text{dist}(x, \supp \mu)$ and consider $y \in \supp \mu$ such that $\verti{x-y}= \text{dist}(x,\supp \mu)$.
	
	We have 
	\bes
	\ba
	(W\ast\mu)(x) \lesssim (E\ast \mu)(x) &\leq \fint_{B(x,r)} (E\ast \mu)(z) \intd{z} \leq 2^N \fint_{B(y,2r)} (E\ast \mu)(z) \intd{z} \\ &\overset{(a)}{\leq} 2^N (E\ast \mu)(y) \lesssim W\ast \mu(y) \leq \alpha-\frac{1}{2}V(y) \leq \alpha + \frac{1}{2}\vertii{V}_{L^{\infty}(B_1)}\text{.} 
	\ea
	\ees
	Here, we rely on Lemma \ref{Enu} for $(a)$.
\end{proof}

\section{A variational inequality}\label{sec var}
This section is strongly inspired by \cite{carrillo2016regularity}. Its purpose is to point out a variational problem, similar to \eqref{var int}, that will allow us to show that $\mu$ has an $L^2$ density. 

We start with an heuristic derivation of a variational inequality involving $W\ast \mu$ when $W(x)=1/\verti{x}^{N-2}$. In this case, $W$ is such that $-\Lap W=N(N-2)\verti{B_1} \delta_0$. This fact, combined with the Euler--Lagrange equation \eqref{el1}, implies that $w=W\ast\mu$ solves the obstacle problem \eqref{obs int}. Assuming for simplicity that $w$ is $C\cs^{\infty}$, we find that, for each  $v \in \mathring{H}^1$ such that $v\geq \alpha- 1/2 V$,  
\be\label{ipp}
\int_{\R^N} \nabla w \cdot \nabla (v-w) \intd{x} = \int_{\R^N}(-\Lap W) (v-w) \intd{x} \geq 0\text{,}
\ee
since $-\Lap w$ is non-negative and supported in $\supp\mu \subset \{w =\alpha- 1/2 V\}$, where $v\geq w$.

In this section, we introduce a bilinear form $a$ (see \eqref{def a}) in relation with $W$. More precisely, $a$ and $W$ will be related in a way that mimics the connection between the Dirichlet form 
\be\label{diri}
(v,w) \mapsto \int_{\R^N} \nabla v \cdot \nabla w \intd{x}
\ee
and the fundamental solution of $-\Lap$ (see \eqref{ipp2} and \eqref{motiv}).

Next, we justify an integration by parts in the vein of \eqref{ipp}, which allows us to show that, when $W$ satisfies \eqref{H1} and \eqref{H2}, $W\ast \mu$ solves a variational inequality similar to the above, where the Dirichlet form \eqref{diri} is replaced with $a$.
 For this purpose, we  rely on Lemmas \ref{fun H1}--\ref{Linfty}, and our approach is slightly different from the ones that have been used to deal with the case $W=1/\verti{x}^{N-2}$ (see, e.g., \cite[Proposition 2.22]{serfaty2015coulomb}). The main issue that forces us to take a different route is the fact that $W\ast \mu$ is not superharmonic. 

 Recall that we have assumed that $\supp \mu \subset B_1$. 
Let $u_0$ be a smooth function such that
\bes
\ba
0 \leq u_0 \leq \alpha\text{, }u_0=\alpha \ \text{in } \supp \mu\text{,} \ \text{and} \ u_0=0 \in B_2^\mathrm{c}\text{.} 
\ea
\ees

We also set $f= 1/2\eta V $, where $\eta$ is a cut-off function such that $\eta=1$ in $B_4$. Since $V \in H^2_{\loc}$, we have $f \in H^2$. By \eqref{obs eq} and the fact that \(W\) is positive, we have
\bes
\ba
& W\ast\mu + f \geq  u_0\text{,} \ \text{Lebesgue a.e.,} 
\\ & W\ast\mu + f = u_0\text{,} \ \mu-\text{a.e.\ in} \ \supp\mu\text{.}
\ea 
\ees
We set
\be\label{def u}
u \coloneq W\ast \mu + f-u_0\text{.}
\ee

We will show that $u$ solves a variational inequality involving the symmetric bilinear form
\be\label{def a}
a \colon \mathring{H}^1 \times \mathring{H}^1 \ni (v,w) \mapsto \int_{\R^N} \frac{\verti{\xi}^2} {m(\xi/\abs{\xi})} \mathscr{F}(v)(\xi) \overline{\mathscr{F}(w)(\xi)} \intd\xi\text{.}
\ee 

We also set  
\be\label{def F}
F(v) \coloneq  \int_{\R^N} \frac{\verti{\xi}^2}{m(\xi/\verti{\xi})}  \mathscr{F}(f-u_0)(\xi)\overline{\mF(v)(\xi)} \intd\xi=a(f-u_0,v)\text{.}
\ee
The bilinear form $a$ is associated with the linear operator $A$ defined, at least for $v\in L^2$ or $v\in \mathring H^1$,  by
\be\label{Aop}
\mathscr{F}(Av)(\xi)= (2\pi)^N \frac{\verti{\xi}^2}{m(\xi/\verti{\xi})} \mathscr{F}(v)(\xi)\text{.}
\ee

We have, by Plancherel's formula,
\be\label{ipp2}
a(v,w)= \int_{\R^N} (Av) w \intd{x}\text{,}
\ee
for each sufficiently smooth $v \in \dot{H}^2$ and $w \in L^2$, where 
\be\label{doth2}
\dot{H}^2 \coloneq \{v \in L^1_{\loc}\text{;} \, \partial_{ij} v \in L^2\text{,} \ \fo 1 \leq i\text{, }j \leq N \}\text{.}
\ee
The connection between $a$ and $W$ is provided by the formula
\be\label{motiv}
A\left(\frac{W}{(2\pi)^N}\right)=\delta_0\text{,}
\ee
which will play a role in what follows. 

We first gather some obvious properties of $a$ and $F$. These properties will be instrumental in the following.  
\begin{lemma}
    $a$ is symmetric, real-valued and satisfies 
    \be\label{a coer}
a(v,v) \approx \vertii{\nabla v}_{L^2}^2\text{,} \ \fo v \in \mathring{H}^1\text{.}
\ee
$F$ is also real-valued. Moreover, 
\be\label{F est}
F \ \text{is linear and continuous on} \ L^2 \ \text{and} \  \mathring{H}^1\text{.}
\ee
(We will denote by $\vertii{F}_{(L^2)^*}$ and $\vertii{F}_{(\mathring{H}^1)^*}$ its corresponding norms.)
\end{lemma}
\begin{proof}
$a$ is indeed symmetric and real-valued since 
\be\label{a re}
\ba
\overline{a(v,w)}=a(w,v)& = \int_{\R^N} \frac{\verti{\xi}^2} {m(\xi/\abs{\xi})} \mathscr{F}(w)(\xi) \overline{\mathscr{F}(v)(\xi)} \intd\xi
\\ &\overset{(a)}{=}  \int_{\R^N} \frac{\verti{-\xi}^2} {m(-\xi/\abs{-\xi})} \overline{\mathscr{F}(w)(-\xi)} \mathscr{F}(v)(-\xi) \intd\xi
\\ &= a(v,w)\text{.}
\ea
\ee
For $(a)$, we rely on the fact that $\mathscr{F}(w)(-\xi)= \overline{\mathscr{F}(w)(\xi)}$, for any real-valued function $w \in \mathring{H}^1$, and on the evenness of $m$. 
The same computation shows that $F$ is real-valued.

The estimate 
\bes
a(v,v) \approx \vertii{\nabla v}_{L^2}^2\text{,} \ \fo v \in \mathring{H}^1\text{,}
\ees
is a straightforward consequence of \eqref{H2}.  

Using \eqref{H2} and the Cauchy--Schwarz inequality, we find that $F$ is continuous on $L^2$ and $\mathring{H}^1$, since $f-u_0$ is in $H^2$. 
\end{proof}
%

We now state the main result of this section,  inspired by \cite[Proposition 3.3]{carrillo2016regularity}. 

\begin{lemma}\label{setup}
	$u$ is the unique solution of the variational inequality
	\be\label{var ineq}
u\ge 0\text{,}\ 	a(u,v-u) \geq F(v-u)\text{,} \ \forall\, \text{$v \in \mathring{H}^1$, $v\geq 0$ a.e.}
	\ee
\end{lemma}
In order to prove Lemma \ref{setup}, we establish the following auxiliary result.
\begin{lemma}\label{fund solution}
	For each $\varphi \in C\cs^{\infty}$, we have
	\bes
	a(W\ast \mu, \varphi) = (2\pi)^N \int_{\R^N}\varphi \intd\mu\text{.}
	\ees
\end{lemma}
 
\begin{proof}
	Set $\mu_{\ve} = \mu \ast \rho_{\ve}$. By Lemma \ref{fun H1}, $W\ast \mu \in \mathring{H}^1$, thus $W\ast \mu_{\ve} \in \mathring{H}^1$ and $ W\ast \mu_{\ve} \to W\ast \mu$ in $\mathring{H}^1$. 
	Therefore, since $a$ is continuous on $\mathring{H}^1$, we find that
	\be\label{a lim}
	a(W\ast \mu_{\ve},\varphi) \to a(W\ast \mu,\varphi)\text{,} \ \text{as} \ \ve \to 0\text{.}
	\ee
	
	On the other hand, we clearly have 
	\be\label{mu lim}
	\int_{\R^N} \varphi(x) \mu_{\ve}(x) \intd{x} \to  \int_{\R^N} \varphi \intd\mu\text{,} \ee 
	since $\varphi \in C\cs^{\infty}$.
	
	Finally, 
		\be\label{smooth eq}
	\ba
	a(W\ast \mu_{\ve}, \varphi)&=\int_{\R^N} \frac{\verti{\xi}^2}{m(\xi/\abs{\xi})}  \mathscr{F}(W\ast \mu_{\ve})(\xi) \overline{\mathscr{F}(\varphi)(\xi)} \intd\xi
	\\ & = \int_{\R^N}\frac{\verti{\xi}^2}{m(\xi/\abs{\xi})} \mathscr{F}(W)(\xi) \mathscr{F}(\mu_{\ve})(\xi) \overline{\mathscr{F}(\varphi)(\xi)} \intd\xi
	\\ & = \int_{\R^N} \mathscr{F}(\mu_{\ve})(\xi) \overline{\mathscr{F}(\varphi)(\xi)} \intd\xi= (2\pi)^N\int_{\R^N} \varphi(x) \mu_{\ve}(x) \intd{x}\text{.}
	\ea
	\ee
	  Passing to the limits in \eqref{smooth eq} yields the desired conclusion,  using \eqref{a lim} and \eqref{mu lim}.
\end{proof}
\begin{proof}[Proof of Lemma \ref{setup}]
	It suffices to show that \eqref{var ineq} holds for each nonnegative $v\in C\cs^{\infty}$. Indeed, this implies the desired conclusion thanks to: (i) the continuity on $\mathring{H}^{1}$ of the forms $a$ and $F$; (ii) the fact that, for each nonnegative $v \in \mathring{H}^1$,   we may find a sequence $(v_n) \subset C\cs^{\infty}$ of nonnegative functions such that $v_n \to v$ in $\mathring{H}^1$.
	
	Let $v \in C\cs^{\infty}$ be such that $v\geq 0$. We have 
	\be\label{eq auv}
	\begin{aligned}
	a(u,v)&= a(W\ast \mu +f-u_0,v) =a(W\ast \mu,v)+a(f-u_0,v)
	\\
	&= a(W\ast \mu,v)+F(v)\text{.}
	\end{aligned}
	\ee
	
	By Lemma \ref{fund solution}, we have $a(W\ast \mu,v)= (2\pi)^N \int_{\R^N} v \intd\mu$ and thus, since $v\geq 0$,
	\be\label{ineq}
	a(u,v) \geq F(v)\text{.}
	\ee
	
	Next, we show that 
	\be\label{eq}
	a(u,u)=F(u)\text{.}
	\ee
	
	For this purpose, we consider $u_{n}= \theta(\cdot/n)(u\ast \rho_{1/n})$, $n \geq 1$, where $\theta$ is a cut-off function such that $\theta=1$ in $B_1$. The sequence $(u_n) \subset C\cs^{\infty}$ is such that $u_n \geq 0$, for each $n\geq 1$, and $u_n \to u$ in $\mathring{H}^1$ as $n\to \infty$ (see, e.g., \cite[proof of Theorem 11.43]{leoni2017first}). By \eqref{eq auv} and Lemma \ref{fund solution}, we have, for each $n\geq 1$,
	\be\label{un}
	a(u,u_{n})= a(W\ast \mu,u_n)+ F(u_{n}) =(2\pi)^N \int_{\R^N} u_n \intd\mu + F(u_{n})\text{.}
	\ee
	
	By Lemma \ref{almost c0}, we have that $u_n(x) \to u(x)$ for $\mu-$a.e.\ $x \in \R^N$ (since $u$ is continuous $\mu$-a.e.\ and $(u_n)$ is defined by mollifying $u$). Since \(u\) is bounded (by Lemma \ref{Linfty} and the definition of \(u\)), we have the existence of $C<\infty$ such that $\verti{u_n(x)}\leq C$ for each $x \in \R^N$, $n\geq 1$. Hence, we have, by dominated convergence, that
	\bes
	\int_{\R^N}u_n \intd\mu \to \int_{\R^N} u \intd\mu\text{,} \ \text{as} \ n  \to \infty\text{.}
	\ees
	But $u=0$, $\mu-$a.e., by definition of $u$ and Lemma \ref{obs eq}. Thus 
	\be\label{mu lim1}
	\lim_n \int_{\R^N} u_n \intd\mu =0\text{.}
	\ee
	Since $u_n \to u$ in $\mathring{H}^1$, we also have
	\be\label{dx lim}
	a(u,u_n) \to a(u,u) \ \text{and} \ F(u_n) \to  F(u)\text{,} \ \text{as} \ n \to \infty\text{.}
	\ee
	Passing to the limits $n\to \infty$ in \eqref{un}, using \eqref{mu lim1} and \eqref{dx lim}, we find \eqref{eq}.

Combining \eqref{ineq} and \eqref{eq},
	we find 
	\bes
	a(u,v-u)= a(u,v)-a(u,u) \geq  F (v) -  F (u) = F(v-u)\text{.} 
	\ees
	
Finally, the uniqueness of $u$ is standard.
Indeed, for any other solution $w \in \mathring{H}^1$ of \eqref{var ineq}, using
\bes
\ba
&a(u,w-u)\geq F(w-u)\text{,}
\\  & a(w,u-w) \geq F(u-w)\text{,}
\ea
\ees 
we obtain that $a(u-w,u-w) \leq 0$.   Using \eqref{a coer}, we find that  $u=w$ a.e.
\end{proof}

\section{Regularity of the solution of the variational inequality \texorpdfstring{\eqref{var ineq}}{ref}}\label{sec reg}

The main result of this section is the following.
\begin{lemma}\label{var reg}
    The solution $u$ of \eqref{var ineq} is in  $\mathring{H}^1 \cap \dot{H}^2$ (see \eqref{doth2} for the definition of  $\dot{H}^2$).
\end{lemma}
\begin{remark}
Let us emphasize that the proof of Lemma \ref{var reg} only relies on the fact that $u$ solves \eqref{var ineq}, and not on its relation with our initial minimization problem. More specifically, the proof shows that, if $w$ solves 
\bes
w \in \mathring{H}^1\text{,} \ w \geq 0 \ \text{a.e.,}
\ a(w,v-w) \geq  G(v-w)\text{,} \ \fo v \in \mathring{H}^1\text{,} \ v\geq 0 \ \text{a.e.,}
\ees
with $G \in (\mathring{H}^1)^*\cap (L^2)^* $, then $w \in \mathring{H}^1 \cap \dot{H}^2$. 
\end{remark}

While Lemma \ref{var reg} is interesting on its own, its relevance to our study arises in connection with Theorem \ref{main}.

\begin{proof}[Proof of Theorem \ref{main} granted Lemma \ref{var reg}]
We have that $u \in \dot{H}^2$, thanks to Lemma \ref{var reg}.  We thus have, for each $v \in L^2$, 
\be \label{a L2}
\verti{a(u,v)}= \verti{\int_{\R^N} \frac{\verti{\xi}^2}{m(\xi/\verti{\xi})} \mF(u)(\xi) \overline{\mF(v)(\xi)} \intd\xi  } \leq C \vertii{D^2 u}_{L^2} \vertii{v}_{L^2}\text{,}
\ee
thanks to the Cauchy--Schwarz inequality and assumption \eqref{H2}. By the  definitions \eqref{def u}  of $u$ and \eqref{def F}  of $F$ , and thanks to Lemma \ref{fund solution}, we have, for each $\varphi \in C\cs^{\infty}$,
\be\label{mu rec}
\ba
a(u,\varphi)= a(W\ast \mu,\varphi) +F(\varphi)
 = (2\pi)^N \int_{\R^N} \ \varphi \intd\mu +F(\varphi)\text{.}
\ea
\ee
Combining \eqref{a L2} with \eqref{mu rec}, and using the continuity of $F$ on $L^2$, we finally find that 
\bes
\verti{\int_{\R^N} \varphi \intd\mu} \leq C \vertii{\varphi}_{L^2}\text{,} \ \fo \varphi \in C\cs^{\infty}\text{.}
\ees
This shows that $\mu \in L^2$.
\end{proof}
In order to prove Lemma \ref{var reg}, we use an approach that is strongly inspired by \cite[Chapter IV, Section 5]{kinderlehrer2000introduction}. It consists of letting $\ve\to 0$ in the penalized variational problem \eqref{pena} below. Since $u$ is the solution of \eqref{var ineq}, it is the unique minimizer of
\bes
E\colon \mathring{H}^1 \ni v \mapsto \frac{1}{2} a(v,v) -F(v)\text{,}
\ees
in $\{ v \in \mathring{H}^1;  v\geq 0 \}$.   We will \enquote{approximate} the variational problem
\be\label{var}
\text{minimize} \  E(v) \ \text{on} \ \{v \in \mathring{H}^1\text{; }  v\geq 0 \}\text{,}
\ee
by the sequence of problems
\be\label{pena}
\text{minimize} \ E_{\ve}(v)\coloneq \frac{1}{2}a(v,v) + \frac{1}{2}\ve \int_{\R^N} (Tv)v \intd{x} +\frac{1}{2\ve} \int_{\{v<0\}} v^2 \intd{x} -  F(v) \ \text{on} \ H^1\text{.}
\ee

In the above, the term $\frac{1}{2\ve} \int_{\{v<0\}} v^2 \intd{x}$ leads, as $\ve \to 0$, to the constraint  $v\geq 0$ in \eqref{var}.
Compared to \eqref{var}, the advantage of \eqref{pena}  is that it is straightforward to obtain regularity results on its solution.   Our aim is to transfer these results to $u$ \textit{via} the following strategy.

\begin{Step}\label{step1}
	We  obtain $\dot{H}^2$ $\ve$-independent estimates on the solution $u_{\ve}$ of \eqref{pena}.
\end{Step} 

\begin{Step}\label{step2}
	We extract, as $\ve \to 0$, a subsequence of $(u_{\ve})$ that  converges to $\tilde{u}$ (in some weak sense), and we show that $\tilde{u}$ is a solution of \eqref{var ineq}. Since $u$ is the unique solution of \eqref{var ineq}, we obtain that $u=\tilde{u}$. The uniform $\dot{H}^2$ estimates on $(u_{\ve})$ obtained in the first step, and the fact that $u_{\ve}$  converges to $u$ yield $u \in \dot{H}^2$.
\end{Step}
\resetstep

We start with Step 1 and first define $T$ by  
\[
\mF(T(v))(\xi)\coloneq \frac{1}{m(\xi/\abs{\xi})}\mF(v)(\xi)\text{, for a.e.\ }  \xi\in\R^N\setminus\{0\}\text{,}
\]
for each $v \in L^2$. The following result is the consequence of computations similar to \eqref{a re}. We omit its proof. 
\begin{lemma}\label{T props}
	\begin{enumerate}[(i)]
		\item For each $u \in L^2$, $Tu$ is real valued and $ \displaystyle \vertii{Tu}_{L^2} \approx \vertii{u}_{L^2}$. 
		
		\item For each $u \in H^1$,  $ \displaystyle \vertii{Tu}_{H^1} \approx \vertii{u}_{H^1}$. 
		
		\item The bilinear form 
		\bes
		L^2 \times L^2 \ni (u,v) \mapsto \int_{\R^N} (Tu) v \intd{x}
		\ees
		is real-valued and symmetric, and we have
		\be\label{L^2 co}
		\int_{\R^N} (Tu) u \intd{x} \approx \vertii{u}_{L^2}^2\text{,} \ \fo u \in L^2\text{.}
		\ee
	\end{enumerate}
	The same holds for $T^{-1}$.
\end{lemma}
	

 The next result yields the approximating sequence $(u_{\ve})$. 
\begin{lemma}
\label{lemma3.2}
	Let $\ve >0 $. There exists a unique minimizer $u_{\ve} \in H^1$ of 
	\be\label{eps en}
	E_{\ve} \colon H^1 \ni v \mapsto \frac{1}{2}a(v,v) + \frac{1}{2}\ve \int_{\R^N} (Tv) v \intd{x} +\frac{1}{2\ve} \int_{\{v<0\}} v^2 \intd{x} -  F(v)\text{.}
	\ee

Moreover, $u_\ve$ is the unique solution of the corresponding Euler--Lagrange equation 
	\be\label{eps pb}
	a(u_{\ve},v) + \ve \int_{\R^N}(Tu_{\ve})v \intd{x} -\frac{1}{\ve}\int_{\R^N}u_{\ve}^{-} v \intd{x} =  F (v)\text{,} \ \fo v \in {H}^1\text{,}
	\ee

Formally, equation \eqref{eps pb} reads 
\bes
Au_{\ve}+\ve Tu_{\ve}-\frac{1}{\ve}u_{\ve}^{-}=F\text{,}
\ees
with $A$ defined in \eqref{Aop}.
\end{lemma}

\smallskip

\begin{proof}[Proof of Lemma \ref{lemma3.2}]	
	\textit{Existence of $u_\ve$}.
 	By \eqref{a coer}, \eqref{L^2 co}, and since $F$ is continuous on $\mathring{H}^1$, there exist constants $0<c_1,c_2<\infty$ such that
	\be\label{E est}
	     c_1\vertii{\nabla v}_{L^2}^2 + c_1 \vertii{v}_{L^2}^2-c_2 \vertii{\nabla v}_{L^2} \leq E_{\ve}(v)\text{,} \ \fo v \in H^1\text{.}
	\ee
	Consider now a minimizing sequence $(v_n) \subset H^1 $ of $E_{\ve}$. This sequence is bounded in $H^1$. Indeed, \eqref{E est} implies that $(\vertii{\nabla v_n}_{L^2})$ is bounded. Since $(\vertii{\nabla v_n}_{L^2})$ is bounded,  we also obtain that $(\vertii{v_n}_{L^2})$ is bounded, again by \eqref{E est}.
	
	Hence, we may extract a subsequence of $(v_n)$ and find $u_\ve\in H^1$ such that $v_n \rightharpoonup u_\ve$ in $H^1$. Since every term in $E_{\ve}$ is weakly l.s.c.\ with respect to $H^1$, we find that $u_\ve$ minimizes $E_{\ve}$.

\smallskip
\noindent
\textit{$u_\ve$ is unique.} 	This follows from the fact that  every term appearing in the definition of $E_{\ve}$ is convex on $H^1$, while $v \to a(v,v)$ (for example) is strictly convex. 
	
	\smallskip
	\noindent
	\textit{$u_\ve$ satisfies  \eqref{eps pb}}.  Let $v \in H^1$. We have
	\be\label{1 dif}
	\ba
	&\frac{\diffd}{\diffd t}\left(\frac{1}{2}a(u_{\ve}+tv, u_{\ve}+tv)\right)_{|t=0} = a(u_{\ve},v)\text{,} \\ & \frac{\diffd}{\diffd t}\left(\frac{1}{2}\int_{\R^N}\left(T(u_{\ve}+tv) \right)(u_{\ve}+tv) \intd{x} \right)_{|t=0} = \int_{\R^N}(Tu_{\ve}) v \intd{x}\text{,}
	\\ &\frac{\diffd}{\diffd t}\left( F(u_{\ve}+tv)\right)_{|t=0} = F(v)\text{.}
	\ea
	\ee
	
	We also have 
	\bes
	\ba
	\int_{\R^N} &\left((u_{\ve}+ tv)^{-}\right)^2 \intd{x} - \int_{\R^N} \left(u_{\ve}^{-} \right)^2 \intd{x} =  \int_{\{u_{\ve}<-tv\}} (u_{\ve}+tv)^2 \intd{x} - \int_{\{u_{\ve}<0\}}u_{\ve}^2 \intd{x}
	\\ & =\int_{\{u_{\ve}<-tv\}} u_{\ve}^2 \intd{x} - \int_{\{u_{\ve}<0\}} u_{\ve}^2 \intd{x}+2t \int_{\{u_{\ve}<-tv\}} u_{\ve}v \intd{x} + o_{t\to 0}(t)
	\\ &= \int_{\{0 \leq u_{\ve}<-tv\}}u_{\ve}^2 \intd{x} - \int_{\{-tv\leq u_{\ve}<0\}} u_{\ve}^2 \intd{x} +2t \int_{\{u_{\ve}<-tv\}} u_{\ve}v \intd{x} + o_{t\to \infty}(t)\text{.}
	\ea
	\ees
	But 
	\bes 
	\verti{ \int_{\{0 \leq u_{\ve}<-tv\}}u_{\ve}^2 \intd{x}- \int_{\{-tv\leq u_{\ve}<0\}} u_{\ve}^2 \intd{x} } \leq 2 t^2 \int_{\R^N} v^2 \intd{x} =  o_{t\to \infty}(t)\text{.}
	\ees
	Therefore, we have
	\bes
	\frac{\int_{\R^N} \left((u_{\ve}+ tv)^{-}\right)^2 \intd{x} - \int_{\R^N} \left(u_{\ve}^{-} \right)^2 \intd{x}}{t}= 2\int_{\{u_{\ve}<-tv\}}u_{\ve}v \intd{x} + o_{t\to 0}(1)\text{,}
	\ees 
	and we may conclude, by dominated convergence, that 
	\be\label{2 dif}
	\frac{\diffd}{\diffd t}\left(  \int_{\R^N} \left((u_{\ve}+ tv)^{-}\right)^2 \intd{x}\right)_{|t=0} = 2 \int_{\{u_{\ve}<0\}}u_{\ve}v \intd{x} = -2 \int_{\R^N}u_{\ve}^{-} v \intd{x}\text{.}
	\ee
	
 By \eqref{1 dif} and \eqref{2 dif}, $t \to E_{\ve}(u_{\ve}+tv)$ is differentiable at $0$. Since it attains its minimum at $0$, by definition of $u_{\ve}$, \eqref{1 dif} and \eqref{2 dif} also yield \eqref{eps pb}.
\end{proof}

We now show $\ve$-independent $\dot{H}^2$ estimates on the solution $u_{\ve}$ of \eqref{eps pb}.
\begin{lemma}\label{uni H2}
    We have 
    \bes
\vertii{D^2 u_{\ve}}_{L^2} \leq C \vertii{F}_{(L^2)^{*}}\text{,} \ \fo \ve>0\text{,}
\ees
where $C<+\infty$ is independent of $\ve$.
\end{lemma}
\begin{proof} 
We claim that it suffices to prove that, for each $\ve>0$, we have
\be\label{sing}
\vertii{\frac{u^{-}_{\ve}}{\ve}}_{L^2} \leq C \vertii{F}_{(L^2)^{*}}\text{,}
\ee
and
\be\label{easy}
\vertii{\ve Tu_\ve}_{L^2}\leq C \vertii{F}_{(L^2)^{*}}\text{.}
\ee
Indeed, granted \eqref{sing} and \eqref{easy}, we have, by \eqref{eps pb}, 
\be\label{a est}
\verti{a(u_{\ve},v)} \leq C \vertii{F}_{(L^2)^{*}} \vertii{v}_{L^2}\text{,} \ \fo v\in H^1\text{.}
\ee
Therefore, we have, for each $v \in H^1$,
\be\label{a id}
\ba
	\int_{\R^N} \nabla u_\ve \cdot \nabla v \intd{x}
		&=
		\frac{1}{(2\pi)^N}\int_{\R^N} \abs{\xi}^2\mF(u_\ve)(\xi)\overline{\mF(v)(\xi)} \intd\xi 
		\\ &=
		\frac{1}{(2\pi)^N}\int_{\R^N} \frac{\abs{\xi}^2}{m(\xi/\verti{\xi})}\mF(u_\ve)(\xi)\overline{m(\xi/\verti{\xi})\mF(v)(\xi)} \intd\xi
		\\ &= 	\frac{1}{(2\pi)^N}\int_{\R^N} \frac{\abs{\xi}^2}{m(\xi/\verti{\xi})}\mF(u_\ve)(\xi)\overline{\mF(T^{-1}(v))(\xi)} \intd\xi
		\\ &= \frac{1}{(2\pi)^N} a(u_\ve,T^{-1}(v))\text{.}
\ea		
\ee	
Hence, we find that 
\be\label{weak form}
\verti{\int_{\R^N} \nabla u_\ve \cdot \nabla v \intd{x}} \leq C \vertii{F}_{(L^2)^{*}} \vertii{T^{-1}(v)}_{L^2} \leq C \vertii{F}_{(L^2)^{*}} \vertii{v}_{L^2}\text{,} \ \fo v \in H^1\text{,}
\ee
using \eqref{a id} and the fact that $T^{-1}$ is a continuous operator on $L^2$ (Lemma \ref{T props}). This yields the desired conclusion.

We now prove \eqref{sing} and \eqref{easy}. In order to obtain \eqref{sing}, we test equation \eqref{eps pb} with 
\bes
v_1\coloneq -T^{-1}\left(\frac{u_{\ve}^{-}}{\ve}\right)\text{.}
\ees
The function $v_1$ is in $H^1$ since $u_{\ve}^{-} \in H^1$ and $T^{-1}$ is a continuous operator on $H^1$, by Lemma \ref{T props}. 
For later use, we recall that  
\be \label{classi}
\nabla(v^{-})= (-\nabla v ) \mathbbm{1}_{\{v<0\}}\text{,}\ \fo  v \in H^1\text{.}
\ee

We have
\be \label{sign a}
\ba
		a(u_\ve, v_1)
		&=
		-\frac{1}{\ve}\int_{\R^N} \frac{\abs{\xi}^2}{m(\xi/\verti{\xi})} 
		\mF(u_\ve)(\xi)\overline{m(\xi/\verti{\xi})\mF(u^{-}_{\ve})(\xi)} \intd\xi
		\\ & =
		-\frac{1}{\ve}\int_{\R^N} \abs{\xi}^2 
		\mF(u_\ve)(\xi)\overline{\mF(u^{-}_{\ve})} \intd\xi\\
		& \overset{(a)}{=}
		-\frac{(2\pi)^N}{\ve}\int_{\R^N} \nabla u_{\ve} \cdot \nabla (u^{-}_{\ve})  \intd{x}
		 \\ &\overset{(b)}{=}   	\frac{(2\pi)^N}{\ve}\int_{\{u_{\ve}<0\}} \verti{\nabla u_{\ve}}^2  \intd{x} \geq 0\text{.}
		\ea
\ee
Here, we use Plancherel's formula for (a) and the formula \eqref{classi} for (b).

We also have 
\be\label{sign T}
		\ba 
	\int_{\R^N}T(u_{\varepsilon}) \ v_1 \intd{x}
	&=-\frac{1}{\ve (2\pi)^N}\int_{\R^N}
	\frac{1}{m(\xi/\verti{\xi})}\mF(u_\ve)(\xi) \overline{m(\xi/\verti{\xi}) \mF(u^{-}_{\ve})(\xi)} \intd\xi
	\\ & =-\frac{1}{\ve}\int_{\R^N} 
	u_{\ve} u^{-}_{\ve} \intd{x}\ge0\text{.}
	\ea 
\ee
By \eqref{sign a}, \eqref{sign T}, and thanks to the equation \eqref{eps pb}, we find that  
\be \label{cl 1}
\ba
 \int_{\R^N} \left(-\frac{u^{-}_{\ve}}{\ve}\right)T^{-1}\left(-\frac{u^{-}_{\ve}}{\ve}\right) \intd{x}&=\int_{\R^N} \left(-\frac{u^{-}_{\ve}}{\ve}\right) v_1 \\ & \leq F(v_1) \\ &\leq \vertii{F}_{(L^2)^{*}} \vertii{v_1}_{L^2} \\ & \leq C \vertii{F}_{(L^2)^{*}} \vertii{\frac{u^{-}_{\ve}}{\ve}}_{L^2}\text{.}
\ea
\ee
Combined with the fact that $\int_{\R^N} (T^{-1}v) v \intd{x} \approx \vertii{v}_{L^2}$, for each $v \in L^2$ (Lemma \ref{T props}), \eqref{cl 1} yields \eqref{sing}.

In order to obtain \eqref{easy}, we test equation \eqref{eps pb} with 
\bes
v_2 \coloneq \ve u_{\ve} \in H^1\text{.}
\ees

We have 
\be\label{v21}
a(u_{\ve},v_2)=\ve a(u_{\ve},u_{\ve}) \geq 0\text{,}
\ee
and
\be\label{v22}
-\frac{1}{\ve}\int_{\R^N}u_{\ve}^{-}v_2 = -\int_{\R^N}u_{\ve}^{-}u_{\ve} \intd{x} \geq 0\text{.}
\ee
Therefore, using \eqref{eps pb}, \eqref{v21}, and \eqref{v22}, we find that
\be \label{cl 2}
\ba
\ve^2 \int_{\R^N} (T u_{\ve}) u_{\ve} \intd{x} &= \ve \int_{\R^N} (Tu_{\ve}) v_2 \intd{x} \\ &\leq F(v_2) \leq \vertii{F}_{(L^2)^*} \vertii{v_2}_{L^2}= \vertii{F}_{(L^2)^*} \vertii{\ve u_{\ve}}_{L^2}\text{.}
\ea
\ee
Since we have that 
\bes
\ve^2 \int_{\R^N} (T u_{\ve}) u_{\ve} \intd{x} \approx \vertii{\ve u_{\ve}}_{L^2}^2 \approx \vertii{\ve Tu_\ve}_{L^2}^2\text{,}
\ees
by Lemma \ref{T props}, \eqref{cl 2} implies that \eqref{easy} holds.
\end{proof}

We extract a converging subsequence from $(u_\ve)_{\ve>0}$ thanks to the following result.
\begin{lemma}\label{extract}
    We have 
    \bes
    \vertii{\nabla u_{\ve}}_{L^2} \leq C \vertii{F}_{(\mathring{H}^1)^*}\text{,} \ \fo \ve>0\text{,}
    \ees
    where $C<\infty$ is independent of $\ve$.
\end{lemma}
\begin{proof}
          We obtain the desired conclusion by testing \eqref{eps pb} with $v=u_{\ve}$, using the same arguments as in the proof of Lemma \ref{uni H2}. Indeed, since $\int_{\R^N} (Tu_{\ve}) u_{\ve}\intd{x} \geq 0$ (by Lemma \ref{T props}) and $-\int_{\R^N}u_{\ve}^{-} u_{\ve} \intd{x} \geq 0$, we find,   using equation \eqref{eps pb}, that 
          \be\label{cl 3}
          a(u_{\ve}, u_{\ve}) \leq F(u_{\ve})\text{.}
          \ee

          Since $a(u_{\ve},u_{\ve}) \approx \vertii{\nabla u_\ve}_{L^2}$ (by \eqref{a coer}) and $\displaystyle F(u_{\ve}) \leq \vertii{F}_{(\mathring{H}^1)^*} \vertii{\nabla u_{\ve}}_{L^2}$, equation \eqref{cl 3} yields Lemma \ref{extract}.
\end{proof}

We now turn to Step \ref{step2}. By Lemma \ref{extract}, we find a function $\tilde{u} \in \mathring{H}^1$ such that, up to a subsequence, 
\be
	\label{Tca}
	u_{\ve_k}\rightharpoonup \tilde{u} \ \text{in} \ \mathring{H}^1\text{.}
\end{equation}

We have 
\be \label{dotH}
\vertii{D^2 \tilde{u}}_{L^2} \leq \liminf_{k  \to \infty} \vertii{D^2 u_{\ve_k}}_{L^2} \leq C\vertii{F}_{(L^2)^*}\text{,}
\ee
using Lemma \ref{uni H2} for the last inequality. 

The final ingredient of the proof of  Lemma \ref{var reg} is the following.
\begin{lemma}\label{Td4}
The function $\tilde{u}$ satisfies \eqref{var ineq}.
\end{lemma}

\begin{proof}[Proof of Lemma \ref{var reg} granted Lemma \ref{Td4}]
Since $u$ is the unique solution of \eqref{var ineq} (Lemma \ref{setup}), Lemma \ref{Td4}  implies that $u=\tilde{u}$ a.e. By \eqref{Tca} and \eqref{dotH}, we thus have that $u \in \mathring{H}^1 \cap \dot{H}^2.$ 
\end{proof}

We now turn to the proof of Lemma \ref{Td4}.  We use the following properties.  Recall that, in   Hilbert space $H$,  an operator $G : H \to (H)^{*}$ is monotone if 
\be\label{def mon}
	\DualityProd{G(v)-G(w)}{v-w}\geq 0\text{, } \fo v\text{, }w \in H\text{,} 
\ee 
and stricly monotone if, in addition, equality in \eqref{def mon} implies that $v=w$. We have 
\begin{lemma}\label{Td6}
	The operator
	\[
		H^1 \ni v \mapsto
		G_\ve(v)\in (H^1)^*\text{,}
	\]
	defined by
	\[
		\DualityProd{G_\ve(v)}{\varphi}
		=
		a(v,\varphi)+\ve \int_{\R^N}(Tv)\varphi-\frac{1}{\ve}\int_{\R^N}v^-\varphi-F(\varphi)\text{,} \ \fo \varphi \in H^1\text{,}
	\]
	is strictly monotone.
\end{lemma}
\begin{proof}
	We notice that 
	\begin{equation}
		\label{Td1}
		\ba 
			a(v-\varphi,v-\varphi)
			\ge C \int_{\R^N}\abs{\nabla(v-\varphi)}^2\ge0\text{,}			
		\ea 
	\end{equation}
	and 
	\begin{equation}
		\label{Td2}
		\ba 
			\int_{\R^N}(T(v-\varphi))(v-\varphi)\ge C\int_{\R^N}\abs{v-\varphi}^2\ge0\text{.}		
		\ea 
	\end{equation}
	Clearly,  equality in \eqref{Td1} or \eqref{Td2} only occurs when $v=\varphi$.
	Thus, it suffices to prove that
	\begin{equation*}
		(-v^-+\varphi^-)(v-\varphi)(x)\ge0\text{, }\fo x\in\R^N\text{.}
	\end{equation*}
This is a direct consequence of the fact that $t\mapsto -t^{-}$ is increasing.
\end{proof}

We now turn to the 
\begin{proof}[Proof of Lemma \ref{Td4}]
		 We first prove that \(\tilde{u}\ge0\) a.e. 
	We argue by contradiction and assume that $\verti{\{u< 0\}}>0$. This implies that there exists $\delta>0$ such that $\verti{\{\tilde{u}< -\delta\}}>0$. By \eqref{Tca}, we have that, up to a subsequence,  $ u_{\ve_{k}} \to \tilde{u}$ a.e.\ on $A\coloneq \verti{\{\tilde{u}< -\delta\}}$, by compactness of the embedding $\mathring{H}^1 \hookrightarrow L^2_{\loc} $ and the converse to the dominated convergence theorem. Therefore, $\frac{1}{\ve_k^2}\left(u_{\ve_{k}}^{-}\right)^2 \to \infty$ a.e.\ on $A$ and we find, by Fatou's Lemma, that
	\bes
	\liminf_\ell 	 \int_A \frac{1}{\ve_\ell^2}\abs{u_{\ve_\ell}^-}^2 = \infty\text{.}
	\ees
This is a  contradiction, since \(\vertii{\frac{1}{\ve}u_\ve^-}_{L^2} \le C\vertii{F}_{(L^2)^*}\), for each $\ve>0$, by \eqref{sing}.
	
	Next, we prove that 
	\be \label{eq goal}
	a(\tilde{u},v-\tilde{u}) \geq F(\tilde{u})\text{,} \ \fo v \in \mathring{H}^1\text{,} \ v \geq 0\text{.} 
	\ee
	
	Indeed, 
	let \(u_\ve\) be the solution of \eqref{eps pb}.
	By Lemma \ref{Td6}, we have
	\begin{equation}\label{Td7}
		\DualityProd{G_\ve(v)}{v-u_\ve}
		\ge
		\DualityProd{G_\ve(u_\ve)}{v-u_\ve}
		=0
		\text{, }
		\fo v\in H^1 \text{.}
	\end{equation}
	Since \( v^-=0 \) when \(v\ge0\), \eqref{Td7} implies
	\begin{equation}\label{Td8}
		a(v,v-u_\ve)+\ve \int_{\R^N}(Tv)(v-u_\ve)\ge F(v-u_\ve)\text{, }
		\fo v\in H^1 \text{, }v\ge0\text{.}
	\end{equation}
	
	For each $\varphi \in C\cs^{\infty}$, since $\mathscr{F}(\varphi)\in \mathscr{S}$ and $\displaystyle \xi \mapsto \frac{1}{\verti{\xi}^2}$ is integrable near the origin, we have
		$\xi\mapsto\displaystyle \mathscr{F}(\varphi)(\xi)/\verti{\xi} \in L^2\text{,}$
	and thus
	\begin{equation}
		\label{Td9}
		\ba 
		\ve\abs*{\int_{\R^N} T(\varphi)(\varphi-u_\ve)}
		&=
		\frac{\ve}{(2\pi)^N}\abs*{\int_{\R^N} 
			\frac{1}{\abs{\xi}m(\xi/\verti{\xi})}\mF(\varphi)(\xi)\overline{\abs{\xi}\mF(\varphi-u_\ve)(\xi)}}\\
		&\leq C \ve \left( \int_{\R^N} \left(\frac{\mF(\varphi)(\xi)}{\abs{\xi}m(\xi/\verti{\xi})}\right)^2 \intd\xi \right)^{1/2} \vertii{\nabla (\varphi-u_{\ve})}_{L^2}
		\\ &\leq C \ve \left( \int_{\R^N} \left(\frac{\mF(\varphi)(\xi)}{\abs{\xi}}\right)^2 \intd\xi \right)^{1/2} \vertii{\nabla (\varphi-u_{\ve})}_{L^2}
		\\  & \hspace{20 pt} \to
		0\text{,}
		\ea
	\end{equation}
	using \eqref{H2} for the last inequality and Lemma \ref{extract} for the limit.
	
	Passing to the limits $\varepsilon \to 0$ in \eqref{Td8}, using \eqref{Td9} and the weak continuity in $\mathring{H}^1$ of $a(v,\cdot)$ and $F$, we have
	\begin{equation}\label{Te1}
		a(v,v-\tilde{u})
		\ge
		F(v-\tilde{u})\text{, }
		\fo v\in C\cs^\infty \text{, }v\ge0\text{.}
	\end{equation}
	
	Using the fact that a nonnegative map \(v\in\mathring{H}^1\) can be approximated with nonnegative maps in \(C\cs^\infty\), we have
		\begin{equation}\label{Te2}
		a(v,v-\tilde{u})
		\ge
		F(v-\tilde{u})\text{, }
		\fo v\in \mathring{H}^1 \text{, }v\ge0\text{.}
	\end{equation}
	
	Minty's Lemma \cite[Lemma 1.5]{kinderlehrer2000introduction} yields the desired conclusion \eqref{eq goal} since the operator
	\[
		\mathring{H}^1 \ni v \mapsto
		G(v)\in (\mathring{H}^1)^\ast \text{,}
	\]
	with 
	\[
		\DualityProd{G(v)}{\varphi}\coloneq a(v,\varphi)-F(\varphi)\text{,} \ \fo \varphi\in \mathring{H}^1 \text{,}
	\]
 is monotone (this can be shown as in Lemma \ref{Td6}). For the convenience of the reader, we reproduce the proof in \cite[Lemma 1.5]{kinderlehrer2000introduction}.
 
 Let $v \in \mathring{H}^1 $ be nonnegative and set \(w_t=(1-t)\tilde{u}+tv\ge0\) for each \(0\le t\le1\). We have
 \be\label{Te3}
 	\ba 
 	t\left(a(w_t, v-\tilde{u})-F(v-\tilde{u}) \right)& =a(w_t,t(v-\tilde{u}))-F(t(v-\tilde{u}))\\ 
 	&= a(w_t,w_t-\tilde{u})-F(w_t-\tilde{u})
 	\\ & \geq 0
 	\text{,}
 	\ea
 	\ee
 	 using \eqref{Te2} for the inequality.
 This implies that
 \begin{equation}\label{Te4}
 	\ba 
 	a((1-t)\tilde{u}+tv, v-\tilde{u})-F(v-\tilde{u})
 	\ge0
 	\text{,} \ \fo 0<t<1\text{.}
 	\ea
 \end{equation}
 Letting $t \to 0$, we obtain 
 \bes
 a(\tilde{u},v-\tilde{u})\ge F(v-\tilde{u})\text{, } 
 \fo v\in \mathring{H}^1 \text{,} \ v\ge0\text{.} \qedhere
 \ees
\end{proof}

\noindent
\textit{Conclusions and further perspectives.}  We have shown that minimizers of $\mathscr{E}$ have $L^2$ density under rather general assumptions. This raises a few natural questions. First, it would be interesting to know whether $\mu \in L^p$ when $V$ is in $W^{2,p}_{\loc}$, $p\neq2$, instead of $H^2_{\loc}$. We may also wonder whether the approach presented in our article could be applied to the fractional and $2D$ counterparts of $\mathscr{E}$.
More precisely, replacing $W$ with 
\bes
W_{s}(x)= \frac{\Psi(x/\verti{x})}{\verti{x}^{N-2s}}\text{, }0<s<1\text{,}
\ees
or with
\bes
W_{2D}(x)=-\log(\verti{x})+ \Psi(x/\verti{x})\text{,}
\ees
when $N=2$, we may aim at regularity results on the minimizers of the corresponding energies, under assumptions similar to \eqref{hv}, \eqref{H1}, and \eqref{H2}. In a forthcoming note \cite{BullionInPrep}, we show that our method can be adapted to the latter case.

\appendix
\section{Appendix. Existence and uniqueness of a minimizer}\label{A}
This appendix is devoted to the proof of the following result, which holds when \eqref{H2} is relaxed to 
\be\label{relax}\tag{HFWR}
\mathscr{F}(W) \geq 0\text{.}
\ee
\begin{proposition}\label{app res}
Assume that \eqref{H1} and \eqref{relax} hold and that $V$ is a non-negative, coercive, and lower semi-continuous.

 Then, there exists a unique minimizer of the functional $\mathscr{E}$. Moreover, this minimizer is compactly supported.
\end{proposition}
We first present a proof of the following fact, arguing as in Mora~\cite[Proposition 3.2]{mora2025nonlocal}.
\begin{lemma}\label{rep fo}
Let $\nu$ be a positive, finite, and compactly supported measure such that $\mathscr{E}(\nu)<\infty$. We have
\bes
\int_{\R^N}W\ast \nu \intd\nu=  \int_{\R^N} \mathscr{F}(W)(\xi) \verti{\mF(\nu)}^2 \intd\xi\text{.}
\ees
\end{lemma}
\begin{proof}

We approximate $\nu$ with the convolutions $\nu_{\ve}\coloneq \nu \ast \rho_{\ve}$, where $\rho \in C\cs^{\infty}$ is a radial mollifier. 

 Since $\nu$ is a finite and compactly supported, $\nu_{\ve}$ is in $C\cs^{\infty}$ and $\mathscr{F}(W\ast \nu_{\ve}) = \mathscr{F}(W) \mathscr{F}(\nu)\mathscr{F}(\rho_{\ve})$. For each $\ve>0$, the Plancherel formula yields
\be\label{smooth}
\int_{\R^N} (W\ast \nu_{\ve}) \intd{\nu_{\ve}} = \int_{\R^N} \mathscr{F}(W)(\xi) \verti{\mF(\nu)(\xi)}^2 \abs{\mF(\rho_{\ve})(\xi)}^2 \intd\xi\text{.}
\ee
We first study the convergence of the left-hand side in \eqref{smooth}.
Since 
    \bes
    \int_{\R^N}(W\ast \nu_{\ve})(x) \nu_{\ve}(x) \intd{x}= \left((W\ast \nu_{\ve})\ast \overset{\curlyvee}{\nu_{\ve}} \right)(0)= \left(W\ast \nu \ast \rho_{\ve} \ast  \overset{\curlyvee}{\nu} \ast \rho_{\ve}\right)(0)\text{,}
    \ees
	where \( \overset{\curlyvee}{\nu}(x)\coloneqq {\nu}(-x)\), we have, by Tonelli's theorem,
	\be\label{rewrite}
	\ba
	 \int_{\R^N}(W\ast \nu_{\ve})(x) \nu_{\ve}(x) \intd{x} &=  \left(W\ast (\rho_{\ve}\ast\rho_{\ve}) \ast \nu  \ast \overset{\curlyvee }{\nu} \right)(0)
	 \\  = &\int_{\R^N}\int_{\R^N} (W\ast \eta_{\ve})(x-y) \intd\nu(x) \intd\nu(y)\text{,} 
	\ea
	\ee
where $\eta_{\ve}= \rho_{\ve}\ast \rho_{\ve}$ is also a radial mollifier with compact support. Since $W$ is continuous outside the origin and $\lim_{x\to 0} W(x) =\infty$, we have that $W\ast \eta_{\ve}(x) \to W(x)$ for every $x\in \R^N$. On the other hand, we have, for each $x\in \R^N$, 
\bes
(W\ast\eta_{\ve})(x) \leq C(E\ast\eta_{\ve})(x) \leq C E(x) \leq C W(x)\text{.}
\ees
Here, we use \eqref{H1} for the first inequality and for the last one, while the second inequality is a consequence of the superharmonicity of $E$ (Lemma \ref{Enu}), since $\eta$ is a radial mollifier. Since $\mathscr{E}(\nu) <\infty$, we have 
\bes
\int_{\R^N}(W\ast \nu_{\ve})(x) \nu_{\ve}(x) \intd x=\int_{\R^N} (W\ast \eta_{\ve})(x-y) \intd\nu(x) \intd\nu(y) \to \int_{\R^N}(W\ast \nu)(x) \intd \nu(x)\text{,}
\ees
by dominated convergence. 

We next study the convergence of the right-hand side in \eqref{smooth}. We have $\mathscr{F}(\rho_{\ve})(\xi) \to 1 $, for each $\xi \in \R^N$. Thanks to the positivity of $\mathscr{F}(W)$, we therefore find that
\bes
\mathscr{E}(\nu)=\liminf \int_{\R^N} W\ast \nu_{\ve} \intd{\nu_{\ve}} \geq  \int_{\R^N} \mathscr{F}(W)(\xi) \verti{\mF(\nu)(\xi)}^2  \intd\xi\text{,}
\ees
by Fatou's lemma.
Thus, $\mathscr{F}(W)\verti{\mathscr{F}(\nu)}^2 \in L^1$ and since $\vertii{\mathscr{F}(\rho_{\ve})}_{L^{\infty}}$ is uniformly bounded, this allows us to obtain 
\bes
\int_{\R^N} \mathscr{F}(W)(\xi) \verti{\mF(\nu)(\xi)}^2 \abs{\mF(\rho_{\ve})(\xi)}^2 \intd\xi \to \int_{\R^N} \mathscr{F}(W)(\xi) \verti{\mF(\nu)(\xi)}^2  \intd\xi\text{,}
\ees
by dominated convergence.
\end{proof}
Before turning to the proof of Proposition \ref{app res}, we state a last useful lemma. 
\begin{lemma}\label{iso}
Let $f$ and $g$ be real analytic functions defined on $\R^N$ such that $f=g$ on a set $A$ of positive Lebesgue measure. Then, we have $f=g$.
\end{lemma}
The proof Proposition \ref{app res} we present below  essentially follows  the arguments used in \cite[Section 2]{mora2019equilibrium}, \cite[Proposition 3.1]{mateu2023explicit}, and \cite[Lemma 2.10, Theorem 2.1]{serfaty2015coulomb}.
\begin{proof}[Proof of Proposition \ref{app res}]
\textit{Existence of a minimizer.} This fact follows from the direct method. Let $(\mu_n)$ be a minimizing sequence of $\mathscr{E}$ satisfying, with no loss of generality, $\sup_k \mathscr{E}(\mu_k) <\infty$. 

For each $\ve>0$, we find $R>0$ such that $V(x) \geq 1/\ve$ in \(B_R^\mathrm{c}\), by coercivity of $V$. We thus have, for each  $n$,
\bes
\int_{B_R^\mathrm{c}}  1/\ve \intd\mu \leq \int_{B_R^\mathrm{c}} V \intd\mu \leq \mathscr{E}(\mu_n) \leq \sup_k \mathscr{E}(\mu_k)\text{,}
\ees
using the positivity of $V$ and $W$ for the second inequality. This reads
\bes
\mu_n(B_R^\mathrm{c}) \leq \ve \sup_k \mathscr{E}(\mu_k)\text{,}\ \fo n\text{,} 
\ees
meaning that the sequence $(\mu_n)$ is tight. By Prokhorov's theorem, we may therefore find $\mu \in \mathscr{P}(\R^N)$ such that, up to extraction, $\mu_n \rightharpoonup \mu$ in the sense of probability measures. Since $\mathscr{E}$ is l.s.c.\ for the weak convergence of probability measures (see, e.g., \cite[Lemma 0.1 and (1.4.4)]{landkof1972foundations}), we find that
\bes
\mathscr{E}(\mu) \leq \liminf_k \mathscr{E}(\mu_k)= \inf \mathscr{E}\text{,}
\ees
and thus $\mu$ is a minimizer of $\mathscr{E}$.

\smallskip
\noindent
\textit{Minimizers of $\mathscr{E}$ are compactly supported}. This is obtained by repeating the arguments in  Serfaty \cite[Proof of Theorem 2.1, Step 1]{serfaty2015coulomb}. 

\smallskip
\noindent
\textit{Uniqueness.} By the above, it suffices to prove that $\mathscr{E}$ is stricly convex on the space of compactly supported probability measures $\nu$ such that $\mathscr{E}(\nu)<\infty$. Let $\nu_{1}, \nu_{2}$ be such probability measures. Setting $\nu_t=(1-t)\nu_1+t\nu_2$, for each $0<t<1$, we have
\bes
\int_{\R^N}(W\ast\nu_t) \intd{\nu_t}= \int_{\R^N} \mF(W)(\xi) \verti{(1-t) \mF(\nu_1)(\xi)+ t\mF(\nu_2)(\xi)}^2 \intd\xi\text{,}
\ees
by Lemma \ref{rep fo}. But $ \mathbb{C} \ni z \mapsto \verti{z}^2$ is  convex, thus we have 
\be\label{convex}
\ba
&\int_{\R^N}(W\ast\nu_t) \intd{\nu_t}=\int_{\R^N} \mF(W)(\xi) \verti{(1-t) \mF(\nu_1)(\xi)+ t\mF(\nu_2)(\xi)}^2 \intd\xi
\\
&\leq (1-t) \int_{\R^N} \mF(W)(\xi)\verti{\mF(\nu_1)(\xi)}^2 \intd\xi + t \int_{\R^N}\mF(W)(\xi)  \verti{\mF(\nu_2)(\xi)}^2 \intd\xi
\\ &=(1-t) \int_{\R^N}W\ast\nu_1 \intd{\nu_1} + t \int_{\R^N}W\ast\nu_2 \intd{\nu_2}\text{,}
\ea
\ee
using \eqref{relax}. 

We now show that if, for some $t\in (0,1)$,  we have equality in \eqref{convex}, then $\nu_1=\nu_2$. Indeed, consider a set   $A\subset\R^N$ of positive Lebesgue measure where $\mF(W)>0$.  If there is equality in \eqref{convex}, then $\mathscr{F}(\nu_1)=\mathscr{F}(\nu_2)$ on $A$, since $z \mapsto \verti{z}^2$ is stricly convex. But $\mF(\nu_{1,2})$ are real analytic functions, since $\nu_{1,2}$ are compactly supported (this is an easy case of the Paley--Wiener--Schwartz theorem). Therefore, Lemma \ref{iso} implies that $\mathscr{F}(\nu_1)=\mathscr{F}(\nu_2)$ everywhere and this yields the desired conclusion.
\end{proof}

\bibliographystyle{plain}
\bibliography{min}

\end{document}